\documentclass{lmcs}
\pdfoutput=1

\usepackage{lastpage}
\lmcsdoi{21}{4}{1}
\lmcsheading{}{\pageref{LastPage}}{}{}%
{Nov.~12,~2024}{Oct.~03,~2025}{}

\usepackage{amsmath}
\usepackage{amsthm}
\usepackage{xcolor}

\newtheorem{proposition}[thm]{Proposition}
\newtheorem{lemma}[thm]{Lemma}
\newtheorem{corollary}[thm]{Corollary}

\newtheorem{alphatheorem}{Theorem}

\theoremstyle{definition}

\newtheorem{conjecture}[thm]{Conjecture}
\newtheorem{question}[thm]{Question}

\usepackage{amssymb}
\usepackage{enumitem}  
\usepackage{stmaryrd}
\usepackage{hyperref}
\usepackage[mathscr]{euscript}
\usepackage{lineno}
\usepackage[utf8]{inputenc}

\setenumerate{label={\normalfont(\textit{\roman*})}, ref={(\textrm{\roman*})},wide}

\newcommand{\fp}[1]{\left< #1 \right> }
\newcommand{\fpnormal}[1]{\langle  #1 \rangle }

\newcommand{\ip}[1]{\left\lfloor #1 \right\rfloor }
\newcommand{\ipnormal}[1]{\lfloor #1 \rfloor }
 
\newcommand{\nint}[1]{\left\lfloor #1 \right\rceil}

\newcommand{\fpa}[1]{\left\lVert #1 \right\rVert_{\RR/\ZZ}}

\newcommand{\floor}[1]{\left\lfloor #1 \right\rfloor}
\newcommand{\bra}[1]{\left( #1 \right)}
\newcommand{\brabig}[1]{\big( #1 \big)}
\newcommand{\branormal}[1]{( #1 )}

\newcommand{\abs}[1]{\left|#1\right|}
\newcommand{\set}[2]{\left\{ #1 \ \middle| \ #2 \right\} }
 
\newcommand{\ceil}[1]{\left\lceil #1 \right\rceil}

\newcommand{\e}{\varepsilon}
\renewcommand{\a}{\alpha}
\renewcommand{\b}{\beta}
\newcommand{\cgamma}{\gamma} % re-defining \c breaks the bibliography AUTHOR = {Cherlin, Gregory and Point, Fran\c{c}oise},

\newcommand{\NN}{\mathbb{N}}
\newcommand{\NNz}{\mathbb{N}}
\newcommand{\QQ}{\mathbb{Q}}
\newcommand{\ZZ}{\mathbb{Z}}
\newcommand{\RR}{\mathbb{R}}

\DeclareMathAlphabet{\mathpzc}{OT1}{pzc}{m}{it}

\newcommand{\cQ}{\mathcal{Q}}

\definecolor{fresh}{HTML}{000000}\definecolor{checked}{HTML}{000000}
\definecolor{external}{HTML}{000000}\definecolor{later}{HTML}{000000}
\definecolor{minor-rev}{HTML}{000000}\definecolor{major-rev}{HTML}{000000}\definecolor{skip}{HTML}{000000}\definecolor{normal}{HTML}{000000}

\renewcommand{\subset}{\subseteq}

\begin{document}

\lmcsorcid{0000-0003-4119-8570}
 \author[J.\ Konieczny]{Jakub Konieczny}
\address{Department of Computer Science, University of Oxford,
Wolfson Building, Parks Road, Oxford OX1 3QD, UK}
\email{jakub.konieczny@gmail.com}
 
\newcommand{\DeltaS}{\operatorname{\mathrlap{\ \circ}\operatorname{\Delta}}}
 \newcommand{\DeltaT}{\operatorname{\mathrlap{\ \phantom{\circ}}\operatorname{\Delta}}}
 
\title[Extensions of Presburger arithmetic by generalised polynomials]{Decidability of extensions of Presburger arithmetic by generalised polynomials}

\begin{abstract}
	We show that the extension of Presburger arithmetic by a quadratic generalised polynomial of a specific form is undecidable. 
\end{abstract}


\maketitle

\newcommand{\Th}{\mathrm{Th}}

\section{Introduction}
\label{sec:intro}

\subsection{Background}
Presburger arithmetic, that is, the first-order theory of natural numbers with addition $\Th(\NN;+)$, is well-known to be decidable. This is in contrast with Peano arithmetic, which also includes multiplication and is well-known to be undecidable. It is not hard to see that if we extend Presburger arithmetic by adding the square function, then we obtain Peano arithmetic; indeed, it is enough to notice that $a \times b = c$ if and only if $(a + b)^2 - a^2 - b^2 = 2c$ (for a survey of relations between different extensions of Presburger arithmetic, see e.g.\ \cite{Korec-2001}). On the other hand, the extension of Presburger arithmetic by an exponential function is decidable \cite{Semenov-1983,CherlinPoint-1986}. We are thus led to ask: Which functions can we add to Presburger arithmetic and still obtain a decidable theory? For the sake of the discussion below, we point out that the first-order theory of integers with addition and order $\Th(\ZZ;<,+)$ is definitionally equivalent with Presburger arithmetic; we will consider extensions of either $\Th(\NN;+)$ or $\Th(\ZZ;<,+)$ depending on which is more natural in the given context.

Results concerning decidability of extensions of Presburger arithmetic connected with the multiplicative structure are surveyed in \cite{Bes-2001}. A wide class of decidable extensions of Presburger arithmetic was investigated by Semenov \cite{Semenov-1979,Semenov-1983}. He showed that $\Th(\NN;+,f)$ is decidable for each $f \colon \NN \to \NN$ which fulfils the condition of \emph{effective compatibility with addition}, which asserts, roughly speaking, that $f$ increases rapidly enough and is periodic modulo any integer. An alternative proof of this result was obtained by Cherlin and Point \cite{CherlinPoint-1986}. In particular, letting $E_k$ denote the exponential function $E_k(n) = k^n$, this result implies that $\Th(\NN;+,E_k)$ is decidable. Similarly, letting $F$ denote the factorial $F(n) = n!$, we have that $\Th(\NN;+,F)$ is decidable.

A \emph{$k$-automatic sequence} is a sequence $f \colon \NNz \to \Sigma$, taking values in a finite alphabet $\Sigma$, such that $f(n)$ can be computed by a finite automaton given the base-$k$ expansion of $n$ as input. For more background on automatic sequences, see \cite{AlloucheShallit-book}. It was shown by B\"uchi \cite{Buchi-1962} (with later corrections, for further discussion see e.g.\ \cite{BHMV-1992,BHMV-1994}) that for each $k \geq 2$ and each $k$-automatic sequence $f \colon \NN \to \Sigma$ the first-order theory $\Th(\NN;+,f)$ is decidable\footnote{Formally, we may identify $f$ with $\abs{\Sigma}$ unary predicates corresponding to the level sets $f^{-1}(x)$ for $x \in \Sigma$.}. In fact, it is known that $\Th(\NN;+,V_k)$ is decidable and $k$-automatic sequences are precisely the sequences whose level sets are definable in $(\NN;+,V_k)$. Here, $V_k \colon \NN \to \NN$ denotes the map which takes $k^i n$ to $k^i$, where $i \in \NNz$, $n \in \NN$ and $k \nmid n$. For further discussion of extensions of Presburger arithmetic related to digital expansions, we refer to the survey paper \cite{BHMV-1992} (cf.\ \cite{BHMV-1994}) and to more recent publications such as \cite{HieronymiSchulz-2020} and \cite{Schulz-2023}. More exotic numeration systems are also studied for instance in \cite
{MousaviSchaefferShallit-2016,
DuMousaviSchaefferShallit-2016,
DuMousaviRowlandSchaefferShallit-2017} and \cite{BaranwalShallit-2019}.

A \emph{Sturmian sequence} is a sequence $s_{\a,\rho} \colon \NN \to \{0,1\}$ given by
\begin{align}\label{eq:intro:def-sturm}
	s_{\a,\rho}(n) &= \ip{\a(n+1) + \rho} - \ip{\a n + \rho},  & n \in \NN,
\end{align}
 where $\a \in (0,1) \setminus \QQ$ and $\rho \in [0,1)$. The first-order theory $\Th(\NN;+,s_{\a,0})$ is decidable if $\a$ is a quadratic irrational and undecidable if the continued fraction expansion is not computable \cite{HieronymiTerry-2018}. A more comprehensive treatment of Sturmian sequences was carried out in \cite{HMOSSS-2022}. The authors prove, among other things, that the first-order theory 
\( \Th( \set{(\NN;+,s_{\a,\rho})}{\a \in (0,1) \setminus \QQ,\ \rho \in [0,1)}) \)
is decidable\footnote{This first-order theory consists, by definition, of all first-order sentences in the language of $(\NN;+,s)$ that are true, for all $\a \in (0,1) \setminus \QQ$ and $\rho \in [0,1)$, when $s$ is interpreted as $s_{\a,\rho}$; see \cite{HMOSSS-2022} for details.}.
 
As already alluded to earlier, for each polynomial sequence $f \colon \ZZ \to \ZZ$ of degree at least $2$, multiplication is definable in the language of $(\ZZ;+,f)$. It follows that $\Th(\ZZ;+,f)$ is undecidable, and so is $\Th(\ZZ;<,+,f)$.

\subsection{New results}
In this paper, we are interested in generalised polynomials, that is, expressions built up from the usual real polynomials with the use of the integer part function $\ip{x}$, addition and multiplication. For instance, the map $f \colon \ZZ \to \ZZ$ given by $f(n) = \ip{\sqrt{2} n^2 \ip{\sqrt{3}n}} + \ip{\pi n} \ip{ n^3/10}$ is a generalised polynomial. Generalised polynomials can also be thought of as a far-reaching generalisation of the notion of a Sturmian sequence. These sequences have been extensively studied, see e.g.\ \cite{BergelsonLeibman-2007,Leibman-2012,AdamczewskiKonieczny-2023-TAMS} and references therein.

In analogy with polynomials, we expect that the extension of Presburger arithmetic by a given generalised polynomial with at least quadratic rate of growth should be undecidable. 

\begin{conjecture}\label{conj:main}
	Let $g \colon \ZZ \to \ZZ$ be a generalised polynomial such that 
\begin{align}\label{eq:conj:g>>n^2}
\liminf_{n \to \infty} \abs{g(n)}/n^2 > 0.
\end{align} Then the first-order theory of $(\ZZ;<,+,g)$ is undecidable.
\end{conjecture}

As a first step in this direction, we verify this expectation for a specific generalised polynomial with quadratic growth. We believe that similar techniques should work with many other explicit examples, but a fully general result remains elusive. Below, for technical reasons, it will be more convenient to work with the nearest integer operation $\nint{\cdot}$, given by $\nint{x} = \ip{x+1/2}$, rather than with $\ip{\cdot}$. 

\begin{alphatheorem}\label{thm:main}
	Let $\a,\b \in \RR$, $\b \neq 0$, be such that $1,\a,\a^2$ are linearly independent over $\QQ$ and let $g \colon \ZZ \to \ZZ$ be the generalised polynomial given by
	\begin{align}\label{eq:def-g=n[an]}
		&& g(n) &= \nint{\b n \nint{\a n}} & (n \in \ZZ).
	\end{align}
	Then the first-order theory of $(\ZZ;<,+,g)$ is undecidable.
\end{alphatheorem}

{
In fact, in the situation of Theorem \ref{thm:main}, the relation $>$ is definable in $(\ZZ;<,+,g)$. The main ingredient needed here is a result essentially due to Vicky Neale \cite{Neale-thesis} asserting that there exists $s \geq 1$ such that each sufficiently large integer $m$ can be represented as $m = n_1 \nint{\a n_1} + n_2 \nint{\a n_2} + \dots + n_s \nint{\a n_s}$. Since $\abs{ g(n) - \b n \nint{\a n} } \leq 1/2$, it follows that each sufficiently large integer $m$ can be represented as $m = g(n_1) + g(n_2) + \dots + g(n_s) + r$ where $r$ is bounded (in fact, we can roughly estimate $0 \leq r < r_0 := \abs{\beta} + s$). Since $g(n) \geq 0$ for all $n \in \ZZ$, we see for all but finitely many $m \in \ZZ$ we have $m > 0$ if and only if there exist $n_1,n_2,\dots,n_s \in \ZZ$ such that for some $0 \leq r < r_0$ we have $m = g(n_1) + g(n_2) + \dots + g(n_s) + r$. Thus, Theorem \ref{thm:main} has the following immediate corollary.
}
\begin{corollary}\label{cor:main}
	Under the same assumption as in Theorem \ref{thm:main}, the first-order theory of $(\ZZ;+,1,g)$ is undecidable.
\end{corollary}
\noindent %FIXED indent
One can also construct non-trivial examples of bounded generalised polynomials. For instance, is not hard to verify that the sequence $g \colon \ZZ \to \{0,1\}$ given by
\begin{align}\label{eq:97:def-g}
	g(n) = 
	\begin{cases}
	1 & \text{if } \fpa{n^2 \a} < \rho,\\
	0 & \text{otherwise,}
	\end{cases}
	& n \in \ZZ,
\end{align}
is a generalised polynomial for any $\a,\rho > 0$. The sequence $g$ given by \eqref{eq:97:def-g} can be thought of as a quadratic analogue of a Sturmian sequence. For sequences of this form, we also have an undecidability result.

\begin{alphatheorem}\label{thm:main-2}
	Let $\a,\rho \in \RR$, be such that $0 < \rho < 1/4$ and $\a$ is not a linear combination of $1$ and $\rho$ with rational coefficients, and let $g \colon \ZZ \to \ZZ$ be the generalised polynomial given by \eqref{eq:97:def-g}. Then the first-order theory of $(\ZZ;<,+,g)$ is undecidable.
\end{alphatheorem}

\subsection{Proof outline}
A key idea behind the proof of Theorem \ref{thm:main} is that the generalised polynomial in question (or indeed any generalised polynomial) coincides with a polynomial on suitably constructed arbitrarily long arithmetic progressions. Indeed, if we fix an integer $L$ and take $m \in \NN$ such that $\fpa{\a m},\fpa{\b m}$ are sufficiently small as a function of $L$ then we can ensure that for $0 \leq l < L$ we have 
\[
	g(l m) = l^2 g(m),
\]
meaning in particular that the restriction of $g$ to the arithmetic progression \[P =\{0,m,2m,\dots,(L-1)m\}\] is a quadratic polynomial. Applying the standard ideas used to define multiplication in terms of the square function, we can then describe triples $a,b,c \in P$ such that $a = km$, $b = l m$ and $c = kl m$. The operation $(m,km,lm) \mapsto klm$ can be thought of as a weak substitute of multiplication. In Section \ref{sec:Mult} we show that the first-order theory of integers equipped with addition and such ``weak multiplication'' is already undecidable. In Section \ref{sec:Proof} we flesh out the details of the argument, which in particular involves defining a suitably large family of arithmetic progressions in the first-order language of $(\ZZ;<,+,g)$.

The proof of Theorem \ref{thm:main-2} is shorter and conceptually simpler. The key idea is that, using equidistribution results for (generalised) polynomial sequences, we are able to define properties which are, roughly speaking, equivalent to the statements that fractional parts of various expressions are sufficiently small. This allows us to define divisibility, which implies undecidability thanks to a classical result of Robinson \cite{Robinson-1949}.

\subsection{Future directions} 
Our main result, Theorem \ref{thm:main}, can no doubt be generalised, and is intended largely as a proof of concept. The assumption that $1,\a,\a^2$ are linearly independent is necessary in order to obtain a suitable equidistribution result in Lemma \ref{lem:distr} and is most likely not necessary for the result to be true. Below we discuss several other potential future directions. 

\subsubsection{Slower growth}
One could consider a stronger variant of Conjecture \ref{conj:main} where $\liminf$ is replaced with $\limsup$ in \eqref{eq:conj:g>>n^2}. The following example shows that this conjecture is false.

\begin{proposition}
	Let $g \colon \ZZ \to \ZZ$ be given by
	\begin{align*}
	g(n) =
		\begin{cases}
		n^2 &\text{if } n = k^{k^m} \text{ for some } m \in \NN;\\
		0 &\text{otherwise}.
		\end{cases}
	\end{align*}
	Then $g$ is a generalised polynomial and $\Th(\ZZ;<,+,g)$ is decidable.
\end{proposition}
\begin{proof}
	It follows from \cite[Thm.\ C]{ByszewskiKonieczny-2018-TAMS} that $g$ is a generalised polynomial. Letting $E_k^+$ denote the ``exponential'' map given by $E_k^+(m) := k^{\abs{m}}$, we see that $g(n) = \ell$ if and only if there exists $m \in \ZZ$ such that $n = E_k^+(E_k^+(m))$ and $\ell = E_k^+(2E_k^+(m))$. Thus, $g$ is definable in $(\ZZ;+,E_k^+)$. It remains to note that $\Th(\ZZ;<,+,E_k^+)$ is decidable because it is definitionally equivalent with $\Th(\NN;+,E_k)$, which we have already noted is decidable.
\end{proof}

Despite examples such as the one above, it seems plausible that the condition \eqref{eq:conj:g>>n^2} can be weakened to a requirement that $g$ grows at least quadratically on a significant portion of the integers. For instance, one could require that there exists $\e > 0$ such that for each sufficiently large $N$ there are at least $\e N$ integers $n$ with $\abs{n} < N$ and $\abs{g(n)} > \e n^2$.

\subsubsection{Hardy field sequences}
Instead of generalised polynomials, one can consider other classes of integer-valued sequences whose behaviour resembles polynomials in some useful sense. A natural class of examples is provided by Hardy field sequences, whose relevant properties are discussed e.g.\ in \cite{Boshernitzan-1994} or \cite{Frantzikinakis-2009}. The definition of a Hardy field function is somewhat technical, but for now suffice it to say that these functions include all logarithmic-exponential functions, i.e., functions defined on $\RR_+$ which can be expressed using the operations $+,\times,\exp,\log$ and real constants. Thus, for instance, $x \log x$, $x^{3/2}$ and $x^x$ are Hardy field functions of $x$.

If $f \colon \RR_+ \to \RR$ is a Hardy field function such that for some $d \in \NN$ we have 
\begin{align}\label{eq:intro:Hardy-growth}
	\lim_{x \to \infty} \frac{f(x)}{x^d} & > 0, 
	& \lim_{x \to \infty} \frac{f(x)}{x^{d+1}} & = 0,
\end{align}
(including the case where the first limit is $\infty$), then on long intervals $f$ can be closely approximated by a degree-$d$ polynomial. Combining this fact with the ideas used in the proof of Theorem \ref{thm:main}, it seems plausible that one could show that for a Hardy field sequence $f \colon \RR_+ \to \RR$ satisfying \eqref{eq:intro:Hardy-growth} with $d \geq 2$, the first-order theory of $(\ZZ;<,+,\nint{f})$ is undecidable. (Here, $\nint{f}$ denotes the sequence given by $\nint{f}(n) = \nint{f(n)}$.) A potentially more interesting question concerns the case where $d = 1$:
\begin{question}
	Let $f \colon \RR_+ \to \RR$ be a Hardy field sequence satisfying 
	\begin{align}\label{eq:intro:Hardy-growth-d=1}
	\lim_{x \to \infty} \frac{f(x)}{x} & = \infty, 
	& \lim_{x \to \infty} \frac{f(x)}{x^{2}} & = 0.
\end{align}
Is the first-order theory of $(\ZZ;<,+,\nint{f})$ decidable?
\end{question}

\subsubsection{Sparse sequences}
Going a step further than Theorem \ref{thm:main-2}, we can find $\{0,1\}$-valued generalised polynomials $g$ that are $0$ almost everywhere but take the value $1$ infinitely often  (i.e., $\frac{1}{N} \sum_{n=1}^N g(n) \to 0$ and $\sum_{n=1}^N g(n) \to \infty$ as $N \to \infty$). In \cite{ByszewskiKonieczny-2018-TAMS} we studied the properties of such generalised polynomials, which we dub \emph{sparse}. We also constructed a number of examples. For instance, the indicator function of the Fibonacci numbers $1_{\mathrm{Fib}}$ is a generalised polynomial (see \cite{ByszewskiKonieczny-2023+} for the state of the art on sparse generalised polynomials arising from linear recurrence sequences). In this particular case, it is known that the first-order theory of $(\NN,+,1_{\mathrm{Fib}})$ is decidable (e.g.\ by \cite{Semenov-1979}). It would be interesting to determine the extent to which this result extends to other sparse $\{0,1\}$-valued generalised polynomial sequences.

\subsubsection{Ranges of sequences}
Let $f \colon \ZZ \to \ZZ$ be a polynomial of degree at least $2$. We have already pointed out that multiplication is definable in the language of $(\ZZ;<,+,f)$. In fact, it is a standard observation that somewhat more is true: Multiplication is definable in the language of $(\ZZ;<,+,f(\ZZ))$, i.e., in Presburger arithmetic extended by the unary predicate corresponding to the image of $f$. As a representative example, we note the for $f(n) = n^3$ we have $f(n+2)-2f(n+1)+f(n) = 6n+6$ so for sufficiently large $n,m$ we have $m = f(n)$ if and only if there exist $m',m''$ such that $m,m',m''$ are consecutive elements of $f(\ZZ)$ and $m'' - 2m' + m = 6n+6$. Thus, $f$ is definable in $(\ZZ;<,+,f(\ZZ))$, which brings us back to the result mentioned before.

One is naturally lead to ask if analogous generalisations apply to our results concerning generalised polynomials. This seems highly plausible (in the case of generalised polynomials with at least quadratic growth). Indeed, for $g$ given by \eqref{eq:def-g=n[an]}, applying a similar discrete differentiation argument as mentioned in the example above, one can define in $(\ZZ;<,+,g(\ZZ))$ a map $g'$ that is reminiscent of $g$, and it seems likely that $\Th(\ZZ;<,+,g')$ should be undecidable. Unfortunately, thus obtained $g'$ would be rather less elegant than $g$, leading to some rather unwieldy computations which we do not pursue at this time.

\subsection*{Acknowledgements}
The author is grateful to James Worrell,
George Kenison,
Andrew Scoones,
Mahsa Shirmohammadi, and
Bertrand Teguia
for many helpful conversations.
The author is supported by UKRI Fellowship EP/X033813/1. For the purpose of open access, the authors have applied a Creative Commons Attribution (CC BY) licence to any Author Accepted Manuscript version arising.

\subsection*{Notation}

We let $\NN = \{1,2,\dots\}$ denote the set of positive integers and put $\NNz = \NN \cup \{0\}$. We also let $\ZZ^* = \ZZ \setminus \{0\}$ denote the set of non-zero integers.

For $x \in \RR$, we let $\floor{x} = \max\set{ n \in \ZZ}{ n \leq x}$ denote the integer part of $x$ (also known as the floor of $x$), $\ceil{x} = - \ceil{-x}$ the ceiling of $x$, and $\nint{x} = \floor{x+1/2}$ the best integer approximation. Similarly, we let $\{x\} = x - \floor{x} \in [0,1)$ and $\fp{x} = x - \nint{x} \in [-1/2,1/2)$. Finally, we let $\fpa{x} = \abs{\fp{x}} = \min\set{\abs{x-n} }{n \in \ZZ} \in [0,1/2]$ denote the circle norm of $x$. 
We will primarily use the operations $\nint{\cdot}$ and $\fp{\cdot}$ in order to avoid the practical difficulties resulting from the fact that the more commonly used operations $\floor{\cdot}$ and $\{\cdot\}$ are discontinuous at $0$.

\color{checked}
\newcommand{\Q}{\mathcal{Q}}
\newcommand{\F}{\mathcal{F}}
\newcommand{\dom}{\operatorname{dom}}

\section{Multiplication}\label{sec:Mult}

In this section we show that the first-order theory of $(\ZZ;+,1)$ extended by ``poor man's multiplication'' is undecidable. Consider a set $\Q \subset \ZZ^4$ satisfying the following properties: %NEEDFIX
\begin{enumerate}[label={\normalfont($\mathsf{Q}_{\arabic*}$)}, ref={\normalfont($\mathsf{Q}_{\arabic*}$)},wide,leftmargin=*]
\item\label{it:Q:A} If $\vec{x} \in \Q$, then there exist $k,l,m \in \ZZ$ such that $\vec{x} = (m,km,lm,klm)$.
\item\label{it:Q:B} For each finite set $F \subset \ZZ^2$ there exists $m \in \ZZ^*$ such that for all $(k,l) \in F$ we have $(m,km,lm,klm) \in \Q$.
\end{enumerate}

\begin{proposition}\label{prop:Pres+Q-undecidable}
	For any set $\Q \subset \ZZ^4$ satisfying \ref{it:Q:A} and \ref{it:Q:B}, the first-order theory of $(\ZZ;+,1,\Q)$ is undecidable.
\end{proposition}
\noindent %FIXED indent
We can use $\Q$ to define a family of partial binary operators $(\times_m)_{m \in \ZZ^*}$ which are specified by the rule $a \times_m b = c$ if and only if $(m,a,b,c) \in \Q$. (Formally, we simply mean that we define a partial function from $\ZZ^3$ to $\ZZ$, but we find the operator notation more evocative.) Note that property \ref{it:Q:A} ensures that for each $(m,a,b) \in \ZZ^* \times \ZZ^2$ there exists at most one $c \in \ZZ$ such that $(m,a,b,c) \in \Q$, so this definition is well-posed. In fact, $c$ is necessarily given by $c = ab/m$, so $a \times_m b = ab/m$ if $ab/m \in \ZZ$ and $(m,a,b,ab/m) \in \Q$, and $a \times_m b$ is undefined otherwise. We let $\dom(\times_m)$ denote the domain of $\times_m$, that is, the set of pairs $(a,b) \in \ZZ^*$ such that there exists $c \in \ZZ^*$ with $(m,a,b,c) \in \Q$. An equivalent way of expressing property \ref{it:Q:B} is that for each finite subset $F$ of $\ZZ^2$, the dilated copy $m F$ of $F$ is contained in the domain of $\times_m$ for at least one $m \in \ZZ^*$.
We point out that the operation $\times_m$ is commutative, associative and distributive in the sense that we have
\begin{align*}
	a \times_m b &=  b \times_m a && (= ab/m),\\
	(a \times_m b) \times_m c &= a \times_m (b \times_m c) && (=abc/m^2),\\
	(a+b) \times_m c &= a \times_m c + b \times_m c && (=(a+b)c/m), 
\end{align*}
assuming that all terms are defined. However, we cannot rule out the possibility that e.g.\ $a \times_m b$ is defined while $b \times_m a$ is not.

For a term $t$ in the language of $(\ZZ; 1,+,-,\times)$ with variables $x_1,x_2,\dots,x_s$ and for $m \in \ZZ^*$ we define $t_m$ to be the expression obtained from $t$ by replacing each instance of $\times$ with $\times_m$. For instance, if 
\begin{align}\label{eq:mult:34:01}
	t = (x_1+x_2) \times (x_2 + x_3 + x_3) + (x_1 \times x_1) \times x_3 + 2,
\end{align}
(using the shorthands $x + y + z = (x+y)+z$ and $2 = 1+1$), then $t_m$ is given by
\begin{align}
	t_m = (x_1+x_2) \times_m (x_2 + x_3 + x_3) + (x_1 \times_m x_1) \times_m x_3 + 2.
\end{align}\label{eq:mult:34:02}
For each polynomial $p \in \ZZ[x_1,x_2,\dots,x_s]$, pick once and for all a term $t^{(p)}$ which represents it. For instance, the term $t$ given by \eqref{eq:mult:34:01} represents the polynomial $p$ given by
\begin{align*}
	p(x_1,x_2,x_3) &= (x_1+x_2)(x_2+2x_3) + x_1^2 x_3 + 2
	\\ &= x_1 x_2 + x_2^2 + 2 x_1 x_3 + 2 x_2 x_3 + x_1^2 x_3 + 2.
\end{align*}
We then let $p_m$ denote the partial function on $\ZZ^s$ represented by $t^{(p)}_m$.

\begin{lemma}\label{lem:form-p_m}
	Let $p \in \ZZ[x_1,x_2,\dots,x_s]$. Then there exists a finite family $\F(p)$ of pairs of polynomials in $\ZZ[x_1,x_2,\dots,x_s]$ such that for $m \in \ZZ^*$ and $n_1,n_2,\dots,n_s \in \ZZ$, the value $p_m(mn_1,mn_2,\dots,mn_s)$ is defined provided that $m q(n_1,n_2,\dots,n_s) \in \dom(\times_m)$ for all $q \in \F(p)$, and is equal to $m p(n_1,n_2,\dots,n_s)$ if defined.
\end{lemma}
\begin{proof}
	We proceed by structural induction. If $p(x_1,x_2,\dots,x_s) = x_i$ for some $1 \leq i \leq s$, then the claim holds with $\F(p) = \emptyset$, and likewise if $p(x_1,x_2,\dots,x_s) = 1$. Thus we may assume that $p$ is either the sum $p^{(1)} + p^{(2)}$ or the product $p^{(1)} \times p^{(2)}$ of simpler polynomials. If $p = p^{(1)} + p^{(2)}$, we take $\F(p) = \F(p^{(1)}) \cup \F(p^{(2)})$ and  note that by the inductive assumption for $p^{(1)}$ and $p^{(2)}$ we have
	\begin{align*}	
		p_m(mn_1,m n_2,\dots, m n_s) 
		&= p^{(1)}_m(mn_1,m n_2,\dots, m n_s) + p^{(2)}_m(mn_1,m n_2,\dots, m n_s)
		\\ &= m p^{(1)}(n_1,n_2,\dots, n_s) + m p^{(2)}(n_1,n_2,\dots,n_s)
		\\ &= m p(n_1,n_2,\dots, n_s),
	\end{align*}
	where defined.
Similarly, if $p = p^{(1)} \times p^{(2)}$ we take $\F(p) = \F(p^{(1)}) \cup \F(p^{(2)}) \cup \{( p^{(1)}, p^{(2)})\}$ and compute that  
	\begin{align*}	
		p_m(mn_1,m n_2,\dots, m n_s) 
		&= p^{(1)}_m(mn_1,m n_2,\dots, m n_s) \times_m p^{(2)}_m(mn_1,m n_2,\dots, m n_s)
		\\ &= m p^{(1)}(n_1,n_2,\dots, n_s) \times_m m p^{(2)}(n_1,n_2,\dots,n_s)
		\\ &= m p(n_1,n_2,\dots, n_s),
	\end{align*}
	where defined.
\end{proof}

\begin{proof}[Proof of Prop.\ \ref{prop:Pres+Q-undecidable}]
	Let $p \in \ZZ[x_1,x_2,\dots,x_s]$ be a polynomial.
	It follows from Lemma \ref{lem:form-p_m} that for $n_1,n_2,\dots,n_s \in \ZZ$ and $m \in \ZZ^*$ such that $m q(n_1,n_2,\dots,n_s) \in \dom(\times_m)$ for all $q \in \F(p)$, we have the equivalence
\begin{align*}
	 p(n_1,n_2,\dots,n_s) &= 0 & \text{ if and only if} &&  
p_m(mn_1,mn_2,\dots,mn_s) &= 0.
\end{align*}	
Since property \ref{it:Q:B} guarantees the existence of admissible $m \in \ZZ^*$, we conclude that $p(n_1,n_2,\dots,n_s) = 0$ if and only if there exists $m \in \ZZ^*$ such that $p_m(n_1,n_2,\dots,n_s) = 0$. The last condition is expressible in the first-order language of $(\ZZ;+,1,\Q)$. Thus, if the first-order theory of $(\ZZ;+,1,\Q)$ was decidable, then it would also be decidable if a given polynomial equation is solvable, which is well-known to be false (see e.g.\ \cite{Matiyasevich-1993}).
\end{proof}

\section{Proof of Theorem \ref{thm:main}}\label{sec:Proof}
\color{checked}

\subsection{Setup}
Throughout this section, we let $g\colon \ZZ \to \ZZ$ be the generalised polynomial given by
\begin{align}\label{eq:27:def-g}
	&& g(n) &= \nint{\b n \nint{\a n}},& n \in \ZZ,
\end{align}
where $\a,\b \in \RR$, $\beta \neq 0$ and $1,\a,\a^2$ are linearly independent over $\QQ$. For concreteness, we assume that $\a,\b > 0$; the argument in the other cases is entirely analogous. Replacing $g(n)$ with $g(an)$ for suitably chosen $a \in \NN$, we may freely assume that $\b \in \ZZ$ or $\b \in \RR \setminus \QQ$.

\color{checked}
\subsection{Differentiation}
A key feature of generalised polynomials which we will make use of is that their iterated discrete derivatives frequently vanish. For a map $f \colon \ZZ \to \ZZ$ and $m \in \ZZ$ we define the discrete derivative $\DeltaT_m f \colon \ZZ \to \ZZ$ by
\begin{align}\label{eq:DeltaN-f}
	\DeltaT_m f(n) &= f(n+m) - f(n),
\end{align}
and the symmetric discrete derivative $\DeltaS_m f \colon \ZZ \to \ZZ$ by
\begin{align}\label{eq:Delta-f}
	\DeltaS_m f(n) &= \DeltaT_m \DeltaT_n f(0) = f(n+m) - f(n)-f(m) + f(0).
\end{align}
As an illustrative example, we point out that if $f$ is a polynomial sequence of degree $d \geq 1$ and $m \neq 0$, then $\DeltaS_m f$ is a polynomial of degree $d-1$ (and if $f$ is constant, then $\DeltaS_m f$ is identically zero). In order to avoid the excessive use of subscripts, and in order to emphasise the symmetry $\DeltaS_m f(n) = \DeltaS_n f(m)$, for each $r \geq 1$ and $n_0,n_1,\dots,n_r \in \ZZ$ we let
\begin{align}\label{eq:Delta^r-f}
	\DeltaS^r f(n_0,n_1,\dots,n_r) &= \DeltaS_{n_r} \dots \DeltaS_{n_1} f(n_0).
\end{align}

\color{checked}
\begin{lemma}\label{lem:D^2=0}
	There exists a constant $C \geq 1$, dependent only on $\a$ and $\b$, such that the following is true. Let $n_0,n_1,n_2 \in \NN$ satisfy $n_0,n_1/n_0,n_2/n_1 \geq C$. Then $\DeltaS^2 g(n_0,n_1,n_2) = 0$ if and only if 
	\begin{enumerate}
	\item\label{cond:lem:D^2=0:I} for all $I \subset \{0,1,2\}$ we have 
	$\nint{\sum_{i \in I} \fp{\a n_i}} = 0$;
	\item\label{cond:lem:D^2=0:II} letting $\gamma_{I} := \sum_{i,j \in I} \fp{\b n_i \nint{\a n_j}}$ for $I \subset \{0,1,2\}$, we have 
	\[ \nint{\gamma_{\{0,1,2\}}} = \nint{\gamma_{\{0,1\}}} + \nint{\gamma_{\{1,2\}}} + \nint{\gamma_{\{2,0\}}}.\]
	\end{enumerate}
\end{lemma}
\noindent %FIXED indent
Henceforth, we let $C$ denote a constant that is admissible in Lemma \ref{lem:D^2=0} above.
\begin{proof}
	We can explicitly compute that for $n,m \in \NN$ we have
\begin{align*}
	\DeltaS g(n,m) &= \nint{\b (n+m) \brabig{\nint{\a n} + \nint{\a m} + e(n,m)}} 
		- \nint{\b n \nint{\a n}} - \nint{\b m \nint{\a m}}
	\\ &= \nint{\b n \nint{\a m}} + \nint{\b m \nint{\a n}} + \nint{\b(n+m)}e(n,m) + f(n,m),   
\end{align*}
where $e(n,m) \in \{-1,0,1\}$ and $f(n,m) \in  \{-4,-3,\dots,3,4\}$ are given by
\begin{align*}
	e(n,m) &= \nint{\a n + \a m} - \nint{\a n} - \nint{\a m},
	\\ f(n,m) &= \nint{\b (n+m) \brabig{\nint{\a n} + \nint{\a m} + e(n,m)}} 
		- \nint{\b n \nint{\a n}} - \nint{\b m \nint{\a m}}
		\\& \phantom{=} - \nint{\b n \nint{\a m}} - \nint{\b m \nint{\a n}} 
		- \nint{\b(n+m)}e(n,m).
\end{align*}
Similarly, we can compute that for $n_0,n_1,n_2 \in \NN$ we have
\begin{align*}
	g(n_0+n_1+n_2) &= g(n_0) + g(n_1) + g(n_2) + \textstyle{\sum_{i \neq j}} \nint{\b n_i \nint{\a n_j}}
	\\& \phantom{=} + \nint{\b(n_0+n_1+n_2)e(n_0,n_1,n_2)} + f(n_0,n_1,n_2),  
\end{align*}
where $e(n_0,n_1,n_2) \in \{-1,0,1\}$ is given by
\begin{align*}
	e(n_0,n_1,n_2) &= \nint{\a n_0 + \a n_1 + \a n_2} - \nint{\a n_0} - \nint{\a n_1} - \nint{\a n_2},
	\\ &= \nint{ \fp{\a n_0} + \fp{\a n_1} + \fp{\a n_2}}
\end{align*}
and $f(n_0,n_1,n_2)$ is given by a similar formula. We stress that $\abs{e(n_0,n_1,n_2)} \leq 1$, which follows from the fact that $-3/2 < \fp{\a n_0} + \fp{\a n_1} + \fp{\a n_2} < 3/2$.

Suppose that $\DeltaS^2 g(n_0,n_1,n_2) = 0$. Combining the formulas above, we see that 
\begin{align*}
	0 = \DeltaS^2 g(n_0,n_1,n_2) &= \beta (n_0+n_1+n_2)e(n_0,n_1,n_2) 
	 - \beta(n_0+n_1)e(n_0,n_1) 
	 \\& \phantom{=} - \beta(n_1+n_2)e(n_1,n_2) - \beta(n_2+n_0)e(n_2,n_0) + O(1).
\end{align*}
Rearranging and dividing by $\beta$ we obtain
\begin{align*}
	0 &= n_2 \bra{ e(n_0,n_1,n_2) - e(n_1,n_2) - e(n_2,n_0) }
	\\& \phantom{=} + n_1 \bra{ e(n_0,n_1,n_2) - e(n_0,n_1) - e(n_1,n_2) }
	\\& \phantom{=} + n_0 \bra{ e(n_0,n_1,n_2) - e(n_0,n_1) - e(n_2,n_0) }
	+ O(1).
\end{align*}
Provided that the constant $C$ was chosen large enough, it follows that the contributions from $n_2$, $n_1$ and $n_0$ above must all be zero, meaning that
\begin{align*}
	 e(n_0,n_1,n_2) - e(n_1,n_2) - e(n_2,n_0) &= 0,\\
	 e(n_0,n_1,n_2) - e(n_0,n_1) - e(n_1,n_2) &= 0,\\
	 e(n_0,n_1,n_2) - e(n_0,n_1) - e(n_2,n_0) &= 0.
\end{align*}
Comparing the three equations above, we infer that for some $e \in \{-1,0,1\}$ we have
\begin{align*}
	e(n_0,n_1) = e(n_1,n_2) = e(n_2,n_0) &= e, & e(n_0,n_1,n_2) = 2e.
\end{align*}
Since $e(n_0,n_1,n_2) \in \{-1,0,1\}$, it follows that $e = 0$. This implies condition \ref{cond:lem:D^2=0:I}. In order to obtain \ref{cond:lem:D^2=0:II}, we expand the equality $\DeltaS^2 g(n_0,n_1,n_2) = 0$ (using the simplifications resulting from \ref{cond:lem:D^2=0:I}) and replace each instance of $\beta n_i \nint{\a n_j}$ with $\fp{\beta n_i \nint{\a n_j}} + \nint{\beta n_i \nint{\a n_j}}$. 

The proof of the converse direction follows along similar lines. Indeed, as alluded to above, assuming condition \ref{cond:lem:D^2=0:I}, an elementary computation shows that the condition $\DeltaS^2 g(n_0,n_1,n_2) = 0$ is equivalent to condition \ref{cond:lem:D^2=0:II}.
\end{proof}

\color{checked}

\subsection{Fractional parts} Our next goal is to express in the first-order logic of the theory under consideration the condition that the circle norm $\fpa{\a n}$ is small for some $n \in \NN$. The following lemma will be helpful.

\color{checked}
\begin{lemma}\label{lem:def-mu}
	Let $n_0,n_1 \in \NN$ satisfy $n_0,n_1/n_0 \geq C$. Then the following conditions are equivalent:
	\begin{enumerate}
	\item\label{cond:lem:def-mu:I} there exists $n_2 \in \NN$ with $n_2 \geq C n_1$ such that $\DeltaS^2 g(n_0,n_1,n_2) = 0$;
	\item\label{cond:lem:def-mu:II} $\abs{ \fp{\a n_0} + \fp{\a n_1} } < 1/2$;
	\end{enumerate}
\end{lemma}
\begin{proof}
	The implication \ref{cond:lem:def-mu:I} $\Rightarrow$ \ref{cond:lem:def-mu:II} is already contained in Lemma \ref{lem:D^2=0}, so we only need to verify the converse implication \ref{cond:lem:def-mu:II} $\Rightarrow$ \ref{cond:lem:def-mu:I}. It follows e.g.\ from the characterisation of distributions of generalised polynomials in \cite{Leibman-2012} that there exists a sequence $(n_{2}^{(k)})_{k=1}^\infty$ (with $n_2^{(k)} \to \infty$ as $k \to \infty$) such that all fractional parts appearing in the conditions in Lemma \ref{lem:D^2=0} and involving $n_2^{(k)}$ tend to $0$ as $k \to \infty$, i.e.,
\begin{align*}
	\fp{\a n_2^{(k)}} &\to 0,  &	\fp{ \b n_2^{(k)} \nint{\a n_2^{(k)}} } &\to 0, \\	
	\fp{ \b n_0 \nint{\a n_2^{(k)}} } &\to 0, &	\fp{ \b n_1 \nint{\a n_2^{(k)}} } &\to 0, \\	
		\fp{ \b n_2^{(k)} \nint{\a n_0} } &\to 0, &	\fp{ \b n_2^{(k)} \nint{\a n_1} } &\to 0.
\end{align*}
As a consequence, for all sufficiently large $k$ we have
\begin{align*}
	\nint{ \fp{\a n_0} + \fp{\a n_1} +  \fpnormal{\a n_2^{(k)}} } &=  \nint{ \fp{\a n_0} + \fp{\a n_1}  } = 0,
\end{align*}
and by the same token 
\begin{align*}
\nint{ \fp{\a n_0} +  \fpnormal{\a n_2^{(k)}} } = \nint{ \fp{\a n_1} +  \fpnormal{\a n_2^{(k)}} } = 0
\end{align*}
Similarly, with $\gamma_{I}^{(k)}$ defined as in Lemma \ref{lem:D^2=0}\ref{cond:lem:D^2=0:II} with $n_2^{(k)}$ in place of $n_2$, for sufficiently large $k$ we have
\begin{align*}
	\nint{\gamma_{\{0,1,2\}}^{(k)}} & = \nint{\gamma_{\{0,1\}}^{(k)}},&
	\nint{\gamma_{\{0,2\}}^{(k)}} & = \nint{\gamma_{\{0\}}^{(k)}} = 0, &
	\nint{\gamma_{\{1,2\}}^{(k)}} & = \nint{\gamma_{\{1\}}^{(k)}} = 0.
\end{align*}
In particular, it follows that we have
\begin{align*}
	\nint{\gamma_{\{0,1,2\}}^{(k)}} & = \nint{\gamma_{\{0,1\}}^{(k)}} 
	+ \nint{\gamma_{\{1,2\}}^{(k)}} + \nint{\gamma_{\{2,0\}}^{(k)}}.
\end{align*}
Thus, Lemma \ref{lem:D^2=0} implies that $\DeltaS^2 g(n_0,n_1,n_2^{(k)}) = 0$ for sufficiently large $k$.
\end{proof}

\color{checked}

Motivated by Lemma \ref{lem:def-mu}, we introduce the following properties
\begin{align}\label{eq:def-mu}
	\mu(n_0,n_1): && (\exists \ n_2)& \ (n_2 \geq C n_1) \wedge (\DeltaS^2g(n_0,n_1,n_2) = 0) && (n_0,n_1 \in \NN);\\
\label{eq:def-psi}
	\psi(m,N): && (\forall \ n )& \ (n \leq N) \Rightarrow \mu(n,m) && (m,N \in \NN).
\end{align}

\begin{lemma}\label{lem:char-small-fpa}
	For each $\e > 0$ there exists $N \in \NN$ such that for all $m \in \NN$, if $\psi(m,N)$ holds, then $\fpa{\a m} < \e$. Conversely, for each $N \in \NN$ there exists $\e > 0$ such that if $\fpa{\a m} < \e$, then $\psi(m,N)$ holds.
\end{lemma}
\begin{proof}
	It follows from Lemma \ref{lem:def-mu} that $\psi(m,N)$ holds if and only if 
	\begin{align}\label{eq:88:01}
		-\frac{1}{2}-\min_{n \leq N} \fp{\a n} < \fp{\a m} < \frac{1}{2}-\max_{n \leq N} \fp{\a n}.
	\end{align}
	It remains to notice that the expression on the left side of \eqref{eq:88:01} tends to $0$ from below as $N \to \infty$, and the analogous statement applies to the expression on the right side.	
\end{proof}

\color{checked}

\subsection{Divisibility}
We will next construct a first-order formula that expresses a condition closely related to divisibility of integers. We will need the following technical lemma, which is the only place where we make use of the assumption that $\a^2$ is linearly independent of $1,\a$ over $\QQ$.

\color{fresh}
\begin{lemma}\label{lem:distr}
	Let $\a \in \RR$ be such that $1,\a,\a^2$ are linearly independent over $\QQ$. Let $\theta = \frac{a + \a b}{c + \a d}$ with $a,b,c,d \in \ZZ$, $d \neq 0$ and $\theta \not \in \QQ$. Then the sequence $\bra{\fp{\a n}, \fp{\a\nint{\theta n}}}_{n \in \ZZ}$ is equidistributed with respect to the measure on $[-1/2,1/2)$ which is the push-forward of the normalised Lebesgue measure on $\{0,1,\dots,d-1\} \times [-1/2,1/2) \times [-1/2,1/2)$ through the map 
\begin{align}\label{eq:transfer}
	(r,x,y) \mapsto \bra{ \fp{dx + \a r }, \fp{b x - c y + \a \theta r - \a \fp{d y + \theta r}} }.
\end{align}
In particular, the point $(0,0)$ belongs to the interior of the closure of the orbit 
\[\set{\brabig{\fp{\a n}, \fp{\a\nint{\theta n}}}}{n \in \ZZ}.\]
\end{lemma}
\begin{proof}
	Let $n \in \ZZ$ and write $n = d n' + r(n)$ where $n' \in \ZZ$ and $0 \leq r(n) < d$. Put also $x(n) = \fp{\a n'}$ and $y(n) = \fp{\theta n'}$. Since $1,\a,\theta$ are linearly independent over $\QQ$, the sequence $\bra{r(n),x(n),y(n)}_{n \in \ZZ}$ is equidistributed in $\{0,1,\dots,d-1\} \times [-1/2,1/2) \times [-1/2,1/2)$. Note that $\a \theta d = a + \a b - \theta c$. We can compute that 
\begin{align*}
	\fp{\a n} &= \fp{d x(n) + \a r(n)},\\
	\fp{\a\nint{\theta n}} &= \fp{\a\theta n - \a\fp{\theta n}}
	\\&= \fp{(a + \a b - \theta c)n' + (\a \theta r(n)) - \a\fp{\theta dn' + \theta r(n)}}
	\\&= \fp{b x(n) - c y(n) + \a \theta r(n) - \a \fp{d y(n) + \theta r(n)}} 
\end{align*}
This proves the first part of the statement. For the second part, we may take $r = 0$ in \eqref{eq:transfer}; it will suffice to show that $(0,0)$ belongs to the interior of the image of $	[-1/2,1/2) \times [-1/2,1/2)$ through the map given by 
\begin{align}\label{eq:transfer-2}
	(x,y) \mapsto \bra{ \fp{dx }, \fp{b x - c y - \a \fp{d y} } }.
\end{align}
This can be accomplished, for instance, by noting that the point $(0,0)$ is non-singular and is mapped to $(0,0)$.
\end{proof}

\color{checked}

\begin{lemma}\label{lem:suff-n|n'}
	Let $n,n' \in \NN$ and let $h \geq 1$. Suppose that for each $\e' > 0$ there exists $\e > 0$ such that for each $m \in \NN$ with $\fpa{\a m} < \e$ there exists $m' \in \NN$ with $\fpa{\a m'} < \e'$ such that $\abs{ \DeltaS g(n,m') - \DeltaS g(n',m)} \leq h$. Then $\nint{\a n'}/\nint{\a n} = n'/n \in \NN$.
\end{lemma}
\begin{proof}
	Our first goal is to show that
\begin{align*}
	\theta &:= \frac{\a n' + \nint{\a n'}}{\a n + \nint{\a n}}
\end{align*}
is rational. For the sake of contradiction, suppose otherwise. It follows from Lemma \ref{lem:distr} that for any $x \in \RR$ with sufficiently small absolute value we can find a sequence $(m_i)_{i=1}^\infty$ such that $\fp{\a m_i} \to 0$ and $\fp{ \a \nint{\theta m_i}} \to x$ as $i \to \infty$. It follows from the assumptions that we can find a sequence $(m_i')_{i=1}^\infty$ with $\fp{\a m_i' } \to 0$ such that for all $i \in \NN$ we have 
\begin{align}\label{eq:82:01}
\abs{ \DeltaS g(n,m'_i) - \DeltaS g(n',m_i)} \leq h.
\end{align}
Restricting our attention to $i$ such that $\fp{\a m_i}$ and $\fp{\a m_i'}$ are sufficiently small with respect to $n$ and $n'$, we conclude from \eqref{eq:82:01} that
\begin{align}\label{eq:82:02}
	m_i' \nint{\a n} + n \nint{\a m_i'} =
	m_i \nint{\a n'} + n' \nint{\a m_i} + O(h). 
\end{align}
Computing $m_i'$ from \eqref{eq:82:02} we obtain
\begin{align}\label{eq:82:03}
	m_i' = \theta m_i + O(h+n+n'). 
\end{align}
Thus, there is a bounded sequence $(c_i)_{i=1}^\infty$ (with $\abs{c_i} = O(h+n+n')$) such that $m_i' = \nint{\theta m_i} + c_i$ for all $i \in \NN$. In particular, we have
\begin{align}\label{eq:82:04}
	\a m_i' = \a \nint{\theta m_i} + \a c_i. 
\end{align}
Passing to a subsequence where $c_i = c$ is constant, taking \eqref{eq:82:04} modulo $1$ and letting $i \to \infty$, we conclude that 
\begin{align}\label{eq:82:05}
	0 \equiv x + \a c \bmod{1}. 
\end{align}
Recall that $x$ was arbitrary, subject only to the constraint that $\abs{x}$ is sufficiently small. Picking any $x \in \RR \setminus (\ZZ + \a \ZZ)$ we reach a contradiction, proving that $\theta \in \QQ$. 

Next, we observe that $\theta$ is an integer. Indeed, if we had $\theta = {a}/{b} \in \QQ \setminus \ZZ$ with $\gcd(a,b) = 1$, then we could run the same argument as above with $x = 1/b$, again leading to a contradiction. Since the argument is fully analogous, we omit the details. Finally, since $\a$ is irrational, $\theta \in \NN$ implies that 
\(	{\nint{\a n'}}/{\nint{\a n}}= {n'}/{n} = \theta. \)
\end{proof}

It is not hard to see that the converse to Lemma \ref{lem:suff-n|n'} also holds. Indeed, if $n,n' \in \NN$ are such that $\nint{\a n'}/\nint{\a n} = n'/n =: t$, then for each sufficiently small $\e > 0$ and each $m \in \NN$ and $m' := tm$ with $\fpa{\a m} < \e$ we have
\begin{align*}
 &&\b m' \nint{\a n} = \b m \nint{\a n'} &= t \b m \nint{\a n},&&\\
 && \b n \nint{a m'} = \b n' \nint{ \a m} &= t \b n \nint{a m},&&\\
 \text{ and hence } && \abs{ \DeltaS g(n,m') - \DeltaS g(n',m) } &\leq 2.&&
\end{align*}
(Specifically, we need $\e < 1/(2t)$ and $\e + t\fpa{\a n} < 1/2$ and $t \e + \fpa{\a n} < 1/2$.)

Motivated by Lemma \ref{lem:suff-n|n'} and the discussion above, we define 
\begin{align}\label{eq:def-delta}
	\delta(n,n'): &&&
	(\exists\ H)\ (\forall\ M')\ (\exists\ M)\ (\forall\ m)\ 
	\psi(m,M) \Rightarrow 
	\\ \nonumber &&& (\exists\ m')\ \psi(m',M') \wedge 
	{ \abs{ \DeltaS g(n,m') - \DeltaS g(n',m)} \leq H } && (n,n' \in \NN). 
\end{align}

\begin{lemma}\label{lem:char-div}
	Let $n,n' \in \NN$. Then $\delta(n,n')$ holds if and only if $\nint{\a n'}/\nint{\a n} = n'/n \in \NN$.
\end{lemma}
\begin{proof}
	Follows by combining Lemma \ref{lem:suff-n|n'} with Lemma \ref{lem:char-small-fpa}.
\end{proof}

For later reference, we point out that for $n' = tn$ the condition $\nint{\a n'}/\nint{\a n} = n'/n $ is equivalent to $t \fpa{\a n} < 1/2$.

\color{checked}

\subsection{Arithmetic progressions}
It follows from Lemma \ref{lem:suff-n|n'} that for each $k \in \NN$, the set of $n \in \NN$ for which $\delta(k,n)$ holds is an arithmetic progression with step $k$, starting at $k$. For $k \in \NN$, let $\ell(k)$ be the endpoint of the aforementioned arithmetic progression; more explicitly,
\( \ell(k) = k \ipnormal{ {1}/\branormal{2\fpa{\a k} }}.\)
Note that $\ell(k)$ is the unique integer $n \in \NN$ such that $\delta(k,n)$ and $\neg \delta(k,n+k)$, and as a consequence the map $\ell \colon \NN \to \NN$ is defined by a first-order formula.

The discussion above allows us to talk about arithmetic progressions in the first-order language under consideration. Indeed, for $k,h \in \NN$ with $h \leq \ell(k)$ we can define the arithmetic progression
\begin{align}\label{eq:def-P}
	P_{m,h} = \set{t m }{ 1 \leq t \leq h/m} 
	= \set{ n \in \NN }{ \bra{m \leq n \leq h} \wedge \delta(m, n)}.
\end{align}
We are specifically interested in arithmetic progressions $P$ such that the restriction of $g$ to $P$ coincides with a classical (as opposed to generalised) polynomial. This property can be expressed easily enough.
\begin{lemma}\label{lem:char-good-prog}
	Let $m,h \in \NN$ be such that $3m \leq h \leq \ell(m)$. The following conditions are equivalent
	\begin{enumerate}
	\item\label{cond:lem:char-good-prog:I} there exists a degree-$2$ polynomial $p$ such that $g(n) = p(n)$ for all $n \in P_{m,h}$;
	\item\label{cond:lem:char-good-prog:II} there exists $a \in \ZZ^*$ such that $\DeltaT_m^2 g(n) = a$ for all $n \in P_{m,h-2m}$.
	\end{enumerate}
\end{lemma}
\begin{proof}
	The implication \ref{cond:lem:char-good-prog:I} $\Rightarrow$ \ref{cond:lem:char-good-prog:II} holds because the application of $\DeltaT_m$ decreases the degree of a polynomial by $1$. Conversely, if \ref{cond:lem:char-good-prog:II} holds and $p$ is the (unique) polynomial of degree at most $2$ with $p(tm) = g(tm)$ for $t \in \{1,2,3\}$, then the leading coefficient of $p$ is $a/(2m^2) \neq 0$ and  a simple inductive argument shows that $g(n) = p(n)$ for all $n \in P_{m,h}$.
\end{proof}

We let $\pi(m,h)$ denote the assertion that the progression $P_{m,h}$ is admissible:
\begin{align}\label{eq:def-pi}
	\pi(m,h): &&& 
	\brabig{3m \leq h \leq \ell(m)} \wedge (\exists\ a \in \ZZ^*) \\
	&&& \nonumber (\forall\ n \in P_{m,h-2m} )\ \DeltaT_m^2 g(n) = a
	&& (m,h \in \NN). 
\end{align}

We will need to know that admissible arithmetic progressions can be arbitrarily long. The following lemma gives a sufficient condition.

\begin{lemma}\label{lem:long-P}
	Let $m, r \in \NN$, $r \geq 2$. Then $\pi(m,rm)$ holds provided that 
	\begin{align}\label{eq:64:01}
	\fpa{\a m} &< \frac{1}{2r} & \fpa{\b m \nint{\a m} } < \frac{1}{2 r^2}. 
	\end{align}
\end{lemma}
\begin{proof}
	Recall from earlier discussion that the first condition in \eqref{eq:64:01} implies that $\ell(m) \geq rm$. For $t \leq r$ we now have
	\begin{align*}
		g(t m) = \nint{\b tm \nint{\a tm} } 
		=   \nint{\b t^2 m \nint{\a m} } = t^2   \nint{\b m \nint{\a m} }.
	\end{align*}
	Thus, for $n \in P_{m,h}$ we have $g(n) = (n/m)^2 g(m)$, which in particular is a polynomial in $n$, as needed. 
\end{proof}

Note that for each $r$, there exist $m \in \NN$ such that \eqref{eq:64:01} holds. This can be inferred for instance from another reference to \cite{Leibman-2012}.

\subsection{Multiplicative quadruples}

With a view towards applying the results of Section \ref{sec:Mult}, we are now ready to define a set of multiplicative quadruples:
\begin{align}\label{eq:def-Q}
	\cQ = \set{ (m,a,b,c) \in \NN^4}
	{ 
	\begin{array}{ll}
	(\exists\ h \in \NN)\ \pi(m,h) \wedge 
	a,b,a+b,m+c \in P_{m,h} \\ \wedge   
	\DeltaS g(m,c) = \DeltaS g(a,b)
	\end{array}
	}
\end{align}
Since for a degree-$2$ polynomial $p(n) = a_2 n^2 + a_1 n_1 + a_0$ we have $\DeltaS p(n,m) = 2a_2 n m$, it follows directly from the definition of $\cQ$ that for all $(m,a,b,c) \in \cQ$ we have $mc = ab$, and consequently $a = k m$, $b = l m$ and $c = klm$ for some $k,l \in \NN$. In fact, by the same argument we see that $(m,a,b,c) \in \cQ$ if and only if $a = km$, $b = l m$ and $c = klm$ for some $k,l \in \NN$ such that $\pi(m,(kl+1)m)$ holds. It follows from Lemma \ref{lem:long-P} that, given $k,l \in \NN$, one can find $m \in \NN$ such that $\pi(m,(kl+1)m)$ holds, and hence the set $\cQ$ given by \eqref{eq:def-Q} satisfies condition \ref{it:Q:B} restricted to $F \subset \NN^2$. Let 
\[
	\cQ_{\pm} = \set{ (m,\sigma a, \tau b, \sigma\tau c) }{ (m,a,b,c) \in \cQ,\ \sigma,\tau = \pm 1}.
\] 
It follows from discussion above that the set $\cQ_{\pm}$ satisfies conditions \ref{it:Q:A} and \ref{it:Q:B}. Hence, we are in position to apply Proposition \ref{prop:Pres+Q-undecidable}, which completes the proof of Theorem \ref{thm:main}.
 
\section{Proof of Theorem \ref{thm:main-2}}

\subsection{Setup}

Throughout this section, we let $g \colon \ZZ \to \{0,1\}$ be the generalised polynomial given by
\begin{align}\label{eq:bohr:def-g}
	g(n) &= 
	\begin{cases}
		1 &\text{if } \fpa{\a n^2} < \rho,\\
		0 &\text{otherwise,}
	\end{cases}
	& n \in \ZZ,
\end{align}
where $\a \in \RR$, $\rho \in (0,1/4)$ and $\a$ is not a linear combination of $1$ and $\rho$ with rational coefficients. 

\subsection{Small fractional parts} Our first step is to express the property that $\fpa{2 \a n}$ and $\fpa{\a n^2}$ are both small. Consider the following property:
\begin{align}\label{eq:bohr:def-mu}
	\mu(m,N): && 
	(\forall\ n \leq N)\ g(n+m) &= g(n) &
	(m,N \in \NN).
\end{align}

\begin{lemma}\label{lem:char-mu}
	For each $\e > 0$ there exists $N \in \NN$ such that for all $m \in \NN$, if $\mu(m,N)$ holds, then $\fpa{2 \a m} < \e$ and $\fpa{\a m^2} < \e$. Conversely, for each $N \in \NN$ there exists $\e > 0$ such that if $\fpa{2\a m} < \e$ and $\fpa{\a m^2 } < \e$, then $\mu(m,N)$ holds. 
\end{lemma}
\begin{proof}
	Let $\e > 0$ and let $N$ be a large integer, to be specified in the course of the argument. If $\mu(m,N)$ holds, then for each $n \leq N$ we have the equivalence
\begin{align}\label{eq:bohr:64:01}
\fpa{\a n^2} &< \rho & \Longleftrightarrow && \fpa{\a n^2 + \fp{2\a m} n + \fp{\a m^2}} < \rho.  
\end{align}
Put $\b := \fp{2\a m}$ and $\cgamma := \fp{\a m^2}$. Condition \eqref{eq:bohr:64:01} implies in particular that there is no $n \leq N$ such that $\fp{\a n^2} \in (\frac{9}{10}\rho,\rho)$ and $\fp{\b n + \cgamma} \in (\frac{1}{10}\rho,\frac{2}{10}\rho)$. It follows that for some $\rho'$ dependent only on $\rho$, using the terminology of \cite{GreenTao-2012}, the sequence 
$\bra{ \bra{\a n^2, \b n + \cgamma}}_{n=1}^N$ fails to be $\rho'$-equidistributed. Assuming, as we may, that $N$ is sufficiently large, it follows (e.g.\ as a very special case of the quantitative Leibman theorem from \cite{GreenTao-2012}) that there exists $k \in \NN$, bounded by a constant $K$ dependent only on $\rho$, such that $\fpa{k \b} \ll_\rho 1/N$. Thus (possibly after replacing $k$ by a factor), we may write $\beta = b/k + x/N$ where $0 \leq b < k$ and $\gcd(b,k) = 1$ and $\abs{x} \ll_\rho 1$.

We claim that $\abs{\cgamma} \leq \e$. Suppose conversely that this was not the case, and for the sake of concreteness assume that $\cgamma \geq \e$. Then, assuming that $N$ is sufficiently large, we use quantitative equidistribution results for polynomial sequences (e.g.\ \cite{GreenTao-2012}) to find $n \in \NN$ such that
\begin{align}\label{eq:bohr:64:02}
	\fp{ \a n^2 } &\in (\rho - \e,\rho),&
	\fp{ \b n } &\in (-\e/2,\e/2), &
	n \abs{x} < (\e/2) N.
\end{align}
This implies $\fpa{\a n^2} < \rho$ and $\fpa{\a n^2 + \b n + \cgamma} > \rho$, contradicting \eqref{eq:bohr:64:01}.

Next, we claim that $b = 0$, or equivalently, $k = 1$. Suppose conversely. Arguing like above, we find $n \in \NN$ such that 
\begin{align}\label{eq:bohr:64:03}
	\fp{ \a n^2 } &\in (\rho - \e,\rho),&
	nb &\equiv 1 \bmod k,& &
	n \abs{x} < (\e/2) N.
\end{align}
It follows that $\fpa{\a n^2} < \rho$ and, assuming, as we may, that $\e < 1/2K$, we have
\begin{align*}
\fpa{\a n^2 + \b n + \cgamma} \geq \rho + \frac{1}{k} - 2\e > \rho,
\end{align*}
which again contradicts \eqref{eq:bohr:64:01}, and completes the proof of the first statement.

We proceed to the proof of the second part of the lemma. Let $N \in \NN$ and put 
\begin{align}\label{eq:bohr:64:def-delta}
	\delta = \min \set{ \abs{\fpa{\a n^2 } - \rho} }{ n \leq N} > 0.
\end{align}
Suppose now that $m$ is such that 
\begin{align}\label{eq:bohr:64:def-eps}
\fpa{ 2 \a m } &< \frac{\delta}{10N} & \text{and} &&
\fpa{\a m^2 } &< \frac{\delta}{10}.
\end{align}
For $n \leq N$ we have
\begin{align*}
	\abs{ \fpa{\a (n+m)^2 - \fpa{\a n^2} }}& \leq \fpa{ 2 \a m n + \a m^2} 
	\\ &\leq N \fpa{2\a m} + \fpa{\a m^2} \leq \delta/5 < \delta,
\end{align*}
and consequently $g(n+m) = g(n)$. It follows that we may take $\e = \delta/10N$.
\end{proof}

Our next step is to express the property that $\fpa{\a m}$ is small, with no restrictions on $\fpa{\a m^2}$. This is easily achieved with the help of $\mu(m,N)$. Define
\begin{align}\label{eq:bohr:def-lambda}
	\lambda(m,N): && 
	(\exists\ n)\ \mu(n,N) \wedge \mu(n+m,N) &
	&
	(m,N \in \NN).
\end{align}
\begin{lemma}\label{lem:char-lambda}
	For each $\e > 0$ there exists $N \in \NN$ such that for all $m \in \NN$, if $\lambda(m,N)$ holds, then $\fpa{2\a m} < \e$. Conversely, for each $N \in \NN$ there exists $\e > 0$ such that if $\fpa{2\a m} < \e$, then $\lambda(m,N)$ holds. 
\end{lemma}
\begin{proof}
	Let $\e > 0$ and pick $N$ sufficiently large that $\mu(m,N)$ implies $\fpa{2\a n} < \e/2$. Then $\lambda(m,N)$ implies $\fpa{2\a m} < \e$, as needed.
	
	Conversely, let $N$ be an integer. Pick $\delta > 0$ sufficiently small that if for some $n$ we have $\fpa{2\a n} < \delta$ and $\fpa{\a n^2} < \delta$, then $\lambda(n,N)$ holds. Let $m$ be an integer such that $\fpa{2\a m} < \delta/2$ and $m > 2/\delta$.	
By Weyl's equidistribution theorem, there exists $n$ such that 
\begin{align}\label{eq:bohr:22:1}
	\fpa{\a n^2} &< \frac{\delta}{2} & \text{and} &&
	\abs{\fp{2\a n} + \fp{\a m^2}/m} < \frac{\delta}{2m}.
\end{align}	
Then we have
\begin{align}\label{eq:bohr:22:2}
	\fpa{2\a n} & \leq \frac{1}{m} < \delta & \text{and} && 
	\fpa{2\a (n+m)} & \leq	\frac{1}{m} + \frac{\delta}{2} < \delta.
\end{align}	
Of course, $\fpa{\a n^2} < \delta$. It remains to estimate $\fpa{\a (n+m)^2}$. We have
\begin{align*}
	\fpa{\a (n+m)^2} & 
	= \fpa{ \a n^2 + m \bra{\fp{2\a n} + \fp{\a m^2}/m}} 
	 < \frac{\delta}{2} + m \frac{\delta}{2m} = \delta.  
\end{align*}	
Thus, we have $\mu(n,N)$ and $\mu(n+m,N)$, as needed. It follows that we can take any $\e > 0$ with $\e < \delta/2$ and $\e < \fpa{\a n}$ for all $n \leq 2/\delta$.
\end{proof}

Next, we express the property that $\fpa{\a n^2}$ is small, which turns out to require more effort. Towards this goal, we introduce the following property.

\begin{align*}\kappa(m,N): &&& 
	(\exists\ M)\ (\forall\ h)\ \bra{g(h) = 1} \wedge \lambda(h,M) \Rightarrow
	\\ &&& (\forall\ L)\ (\exists\ n)\ \lambda(n,L) \wedge \mu(n,N) \wedge \bra{g(h+m+n)=1}
	&&(m,N \in \NN).
\end{align*}

\begin{lemma}\label{lem:char-kappa}
	For each $\e > 0$ there exists $N \in \NN$ such that for all $m \in \NN$, if $\kappa(m,N)$ holds, then $\fpa{\a m^2} < \e$. Conversely, for each $N \in \NN$ there exists $\e > 0$ such that if $\fpa{\a m^2 } < \e$, then $\kappa(m,N)$ holds. 
\end{lemma}
\begin{proof}
	Let $\e > 0$, and let $N$ be such that $\mu(n,N)$ implies $\fpa{\a n^2} < \e/2$ (which exists by Lemma \ref{lem:char-mu}). Pick $m \in \NN$ such that $\kappa(m,N)$ holds; we aim to show that $\fpa{\a m^2} < \e$. For the sake of contradiction, suppose that $\fp{\a m^2} \geq \e$ (the case where $\fp{\a m^2} \leq -\e$ is analogous). Let $M$ be any integer admissible for $\kappa(m,N)$, and let $h$ be an integer such that
\begin{align}\label{eq:bohr:58:1}
	\fp{\a h^2} &\in \bra{\rho-\frac{\e}{8},\rho},
	& \fp{2\a h} & \in \bra{0,\frac{\e}{8m}} &&
	\lambda(n,M), 
\end{align}
which exists by Weyl's equidistribution theorem. 
(Note that the second and the third condition in \eqref{eq:bohr:58:1} can be satisfied by taking $h$ such that $\fp{\a h}$ is sufficiently small and positive.) Bearing in mind Lemma \ref{lem:char-lambda}, we infer from $\kappa(m,N)$ that there exists a sequence of integers $(n_i)_{i=1}^\infty$ such that 
\begin{align}
	\label{eq:bohr:58:2a}\lim_{i \to \infty} \fpa{\a n_i} &= 0, \\
	\label{eq:bohr:58:2b}\fpa{\a n_i^2} &< \e/2 & \text{for all } i \in \NN,\\
	\label{eq:bohr:58:2c}\fpa{\a (h+m+n_i)^2} &< \rho &	\text{for all } i \in \NN.
\end{align}
Passing to a subsequence, we may assume that $\fp{\a n_i^2}$ converges to some limit $\beta \in [-\e/2,\e/2]$ as $i \to \infty$. 
Letting $i \to \infty$ in \eqref{eq:bohr:58:2c} we arrive at a contradiction: 
\begin{align*}
	\rho &\geq \lim_{i \to \infty} \fpa{\a (h+m+n_i)^2} 
	= \fpa{\a (h+m)^2 + \beta} 
	\\& = \fpa{ \a h^2 + \fp{2\a h}m + \a m^2 + \beta} \geq 
	\rho - \frac{\e}{8} - m \frac{\e}{8m} + \e - \frac{\e}{2} = \rho + \frac{\e}{8}.
\end{align*}

We now proceed to the proof of the converse implication. Let $N$ be an integer and let $\e > 0$ be sufficiently small that $\fpa{\a n^2} < 2\e$ and $\fpa{2\a n} < 2\e$ imply $\mu(n,N)$. Pick $M$ sufficiently large that $\lambda(h,M)$ implies that $\fpa{2\a h} < \e/m$, and let $h$ satisfy $g(h) = 1$ and $\lambda(h,M)$ (and hence in particular $\fpa{2\a h} < \e/m$). We need to show that there exists a sequence of integers $(n_i)_{i=1}^\infty$ such that
\begin{align}
	\label{eq:bohr:58:3a}\lim_{i \to \infty} \fpa{\a n_i} &= 0, \\
	\label{eq:bohr:58:3b}\mu(n_i,N) &\text{ is true}& \text{for all sufficiently large } i \in \NN,\\
	\label{eq:bohr:58:3c}\fpa{\a (h+m+n_i)^2} &< \rho &	\text{for all sufficiently large } i \in \NN.
\end{align}
(Note that in \eqref{eq:bohr:58:3b} and \eqref{eq:bohr:58:3c} it is enough to consider sufficiently large $i$ since we can always pass to a subsequence and achieve the same condition for all $i$.) By Weyl's equidistribution theorem, we can find a sequence $(n_i)_{i=1}^\infty$ such that
\begin{align}
	\label{eq:bohr:58:4a}\lim_{i \to \infty} \fp{\a n_i} &= 0, \\
	\label{eq:bohr:58:4b}\lim_{i \to \infty} \fp{\a n_i^2} &= - \fp{2 \a h m + \a m^2}.
\end{align}
By \eqref{eq:bohr:58:4b} can estimate
\begin{align*}
	\lim_{i \to \infty} \fpa{\a n_i^2} & < \frac{\e}{m} m + \e = 2\e,
\end{align*}
and hence \eqref{eq:bohr:58:3b} follows. It remains to deal with \eqref{eq:bohr:58:3c}. We can compute
\begin{align*}
	\lim_{i \to \infty} \fpa{\a (h+m+n_i)^2} =
	\fpa{\a (h+m)^2+ \lim_{i\to\infty} \fp{\a n_i^2} } = \fpa{ \a h^2 } < \rho,  
\end{align*}
as needed.
\end{proof}
\noindent %FIXED indent
The next property that we would like to express is that $\a n^2$ and $\a m^2$ are close modulo $1$. With this goal in mind, we define
\begin{align*}\label{eq:bohr:def-nu}
	\nu(m,\tilde{m},N): &&&
	(\exists\ L)\ (\forall\ n)\ \lambda(n,L) \wedge \kappa(m+n,L) \Rightarrow \\ 
	&&& \kappa(\tilde{m}+n,N) &
	&& (m,\tilde{m},N \in \NN).
\end{align*}
\begin{lemma}\label{lem:char-nu}
	For each $\e > 0$ there exists $N \in \NN$ such that for all $m,\tilde{m} \in \NN$, if $\nu(m,\tilde{m},N)$ holds, then $\fpa{\a (m^2 - \tilde{m}^2) } < \e$. Conversely, for each $N \in \NN$ there exists $\e > 0$ such that if $\fpa{\a (m^2 - \tilde{m}^2) } < \e$, then $\nu(m,\tilde{m},N)$ holds. 
\end{lemma}
\begin{proof}
By Lemma \ref{lem:char-lambda}, condition $\nu(m,\tilde{m},N)$ is equivalent to the statement that for each sequence $(n_i)_{i=1}^\infty$ such that $\fpa{\a n_i} \to 0$ and $\fpa{\a (m+n_i)^2} \to 0$ as $i \to \infty$ we have $\kappa(\tilde m + n_i,N)$ for all sufficiently large $i$. Consider any sequence $(n_i)_{i=1}^\infty$ such that $\fpa{\a n_i} \to 0$. Then 
\begin{align*}
\fp{\a (m+n_i)^2 - \a m^2 - \a n_i^2} & \to 0 & \text{as } i \to \infty,\\
\fp{\a (\tilde{m}+n_i)^2 - \a \tilde{m}^2 - \a n_i^2} & \to 0 & \text{as } i \to \infty.
\end{align*}
Thus, if $\e > 0$ and $N$ is sufficiently large that $\kappa(m,N)$ implies $\fpa{\a m^2} < \e/3$, then (taking any sequence $(n_i)_{i=1}^\infty$ as above) $\nu(m,\tilde m,N)$ implies 
\begin{align*}
	\fpa{\a (m^2-\tilde{m}^2)} & \leq \lim_{i \to \infty} 
	\fpa{\a (m+n_i)^2 - \a m^2 - \a n_i^2} 
	\\ & + \fpa{\a (\tilde{m}+n_i)^2 - \a \tilde{m}^2 - \a n_i^2} 
	\\ &+ \fpa{\a(m+n_i)^2} + \fpa{\a(\tilde{m}+h_i)^2}
	\leq \e/3.
\end{align*}
Conversely, if $N \in \NN$ and $\e > 0$ is sufficiently small that $\fpa{\a m^2}< 3\e$ implies $\kappa(m,N)$, then for any $m,\tilde m$ with $\fpa{\a (m^2 - \tilde{m}^2) } < \e$ and any sequence $(n_i)_{i=1}^\infty$ as above we have
\begin{align*}
	\fpa{\a(\tilde{m}+n_i)^2} 
	& \leq \lim_{i \to \infty} 
	\fpa{\a (m+n_i)^2 - \a m^2 - \a n_i^2} 
	\\ & + \fpa{\a (\tilde{m}+n_i)^2 - \a \tilde{m}^2 - \a n_i^2} 
	\\ &+ \fpa{\a(m+n_i)^2} + \fpa{\a (m^2-\tilde{m}^2)}
	\leq \e,
\end{align*}
and hence $\kappa(\tilde m, N)$ holds.
\end{proof}

\subsection{Divisibility}
Finally, we are able to define divisibility. Consider the relation
\begin{align*}\delta(m,\tilde{m}): &&&
	(\forall\ N)\ (\exists\ L)\ (\forall\ n)\ \nu(m+n,m,L) \wedge \kappa(n,L)  \Rightarrow \\ 
	&&& \nu(\tilde{m}+n,\tilde m,N) &
	&& (m,\tilde{m}\in \NN).
\end{align*}
\begin{lemma}\label{lem:char-delta}
	For each $m,\tilde m \in \NN$ we have $\delta(m,\tilde m)$ if and only if $m \mid \tilde m$.
\end{lemma}
\begin{proof}
	Condition $\delta(m,\tilde m)$ is equivalent to the statement that for each $\e > 0$, for each sequence $(n_i)_{i=1}^\infty$, if $\fpa{\a((m+n_i)^2 - m^2)} \to 0$ and $\fpa{\a n_i^2} \to 0$ as $i \to \infty$, then $\fpa{\a ((\tilde m + n_i)^2 - \tilde m^2)} \leq \e$ for all sufficiently large $i$. Note that the condition that $\fpa{\a((m+n_i)^2 - m^2)} \to 0$ and $\fpa{\a n_i^2} \to 0$ is equivalent to $\fpa{2 \a m n_i} \to 0$ and  $\fpa{\a n_i^2} \to 0$ as $i \to \infty$. If $\tilde{m} = k m$ is a multiple of $m$, then for any sequence $(n_i)_{i=1}^\infty$ like above we have $\fpa{2 \a \tilde{m} n_i} \leq k \fpa{2 \a m n_i} \to 0$ as $i \to \infty$. Conversely, if $\tilde m/m = a/b$ with $\gcd(a,b) = 1$, then we can find a sequence $(n_i)_{i=1}^\infty$ with $\fpa{2 \a m n_i} \to 0$,  $\fpa{\a n_i^2} \to 0$, and $\fpa{2 \a \tilde{m} n_i} \to 1/b$ as $i \to \infty$.
\end{proof}

Lemma \ref{lem:char-delta} implies that divisibility is definable in the first-order theory under consideration. It is a classical \cite{Robinson-1949} result that multiplication can be defined in $(\ZZ;<,+,\mid)$, meaning in particular that the corresponding first-order theory is undecidable. This completes the proof of Theorem \ref{thm:main-2}. 
\bibliographystyle{alphaurl}
\bibliography{bibliography}

\end{document}